\documentclass[12pt,reqno]{amsart}
\usepackage{fullpage}
\usepackage{times}
\usepackage{colonequals}
\usepackage{amsmath,amssymb,amsthm,url}
\usepackage{xcolor}
\usepackage[utf8]{inputenc}
\usepackage[english]{babel}
\usepackage{comment}
\usepackage{bbm}
\usepackage{enumerate}
\usepackage{bm}
\usepackage{graphicx}
\usepackage{mathrsfs}
\usepackage[colorlinks=true, pdfstartview=FitH, linkcolor=blue, citecolor=blue, urlcolor=blue]{hyperref}

\newtheorem{thm}{Theorem}[section]
\newtheorem{lem}[thm]{Lemma}

\theoremstyle{definition}

\newtheorem{defn}[thm]{Definition}
\newtheorem{example}[thm]{Example}
\newtheorem{rem}[thm]{Remark}

\newcommand{\Q}{\mathbb{Q}}
\newcommand{\R}{\mathbb{R}} 
\newcommand{\N}{\mathbb{N}}
\newcommand{\Z}{\mathbb{Z}}

\title{On the fractional parts of certain sequences of $\xi \alpha^{n}$}
\author{Xiang Gao}
\address{Department of Mathematics, Hubei Key Laboratory of Applied Mathematics, Hubei University, Wuhan 430062, China}
\email{gaojiaou@gmail.com}
\author{Chi Hoi Yip}
\address{Department of Mathematics, Hong Kong University of Science and Technology, Clear Water Bay, Hong Kong}
\email{machyip@ust.hk}

\begin{document}

\begin{abstract}
Assume that $\alpha>1$ is an algebraic number and $\xi\neq0$ is a real number. We are concerned with the distribution of the fractional parts of the sequence $(\xi \alpha^{n})$. Under various Diophantine conditions on $\xi$ and $\alpha$, we obtain lower bounds on the number of occurrences for which the fractional part of the sequence $(\xi \alpha^{n})_{n\geq1}$ falls into a prescribed region $I\subset [0,1]$, extending several results in the literature. As an application, we show that the Fourier decay rate of some self-similar measures is logarithmic, generalizing a recent result of Varj\'{u} and Yu.  
\end{abstract}

\maketitle

\section{Introduction}
A real number $\xi\in [0,1]$ is said to be normal to base $\alpha\in\mathbb{N},\alpha> 1$, if in the base $\alpha$ expansion of $\xi=0.\xi_1\xi_2\cdots \xi_n \cdots$, every combination of $k$ digits occurs with the proper frequency for each $k\in\mathbb{N}$. It is well-known that this is equivalent to the statement that the sequence $(\xi \alpha^n )_{n\geq1}$ is uniformly distributed modulo one.  Namely,
$$\lim\limits_{N\rightarrow\infty}\frac{1}{N}\#\{1\leq n \leq N: \xi \alpha^{n}\!\!\!\!\pmod 1\in I\}=|I|,$$
for any subinterval $I$ of $[0,1]$, here we denote the length of $I$ by $|I|$. For a general real number $\alpha>1$, there are two classical metric results as follows. For any fixed $\alpha>1$, Koksma \cite{K35} proved that the sequence $(\xi\alpha^n)_{n\geq1}$ is uniformly distributed modulo one for almost all real numbers $\xi$. On the other hand, Weyl proved that for any fixed $\xi \neq 0$, the sequence $(\xi\alpha^n)_{n\geq1}$ is uniformly distributed modulo one for almost all real numbers $\alpha>1$. The reader may refer to \cite{B12} for more background. 

However, for a specific pair $(\xi,\alpha)$, the distribution of $(\xi \alpha^{n})_{n\geq1}$ modulo one is very far from being understood, except in very few cases. For instance, we do not know whether the sequence $((\frac{3}{2})^n)_{n\geq1}$ is uniformly distributed modulo one. Indeed, it is not even known to be dense. Vijayaraghavan \cite{V40} in $1940$ first showed that there are infinitely many limit points of the sequence $((\frac{p}{q})^n)_{n\geq1}$, where $ p,q $ are relatively prime integers with $p >q \geq 2$. In $1968$, Mahler \cite{M68} conjectured that there does not exist a nonzero real number $\xi$ such that
\[
\Bigl\{\xi\Bigl(\frac32\Bigr)^n\Bigr\}<\frac12
\]
for all positive integers $n$. This is known as Mahler's $3/2$ problem.
In $1995$, Flatto, Lagarias, and Pollington \cite{FLP95} showed that the gap between the largest and the smallest limit point is at least $\frac{1}{p}$. Similar results were proved by Pisot \cite{P38} for $(\xi \alpha^n)_{n \geq 1}$ when $\alpha>1$ is a special algebraic number, and some strengthened results were obtained by Dubickas \cite{D060, D061} for all algebraic numbers $\alpha>1$. 

In this paper, we shall estimate the number of $1\leq n \leq N$ such that  $\xi\alpha^n$ modulo one falls into a prescribed region $I\subseteq [0,1]$. This counting problem can be viewed as a quantitative version of the above-mentioned results. Moreover, this combinatorial problem is closely related to the behaviors of digital expansions of numbers. For example, when $\alpha=3$, the number of terms that the sequence $(\xi\alpha^n)_{n\geq1}$ fall into the region $I=[\frac{1}{3},\frac{2}{3})$ is exactly equal to be the number of digits $1$ in the base-$3$ expansion of $\xi$.

There is currently much interest in the theory of Fourier transforms of self-similar measures, see the definition in Section~\ref{sec:app}. Denote the reciprocal of the contraction ratio of self-similar measure by $\alpha$, we know that the Fourier analytic properties of self-similar measures have close connections with the distribution of fractional part of the sequence \!$(\xi\alpha^n)_{n\geq1}$, see for example \cite{DFW07, E39, S63}; for more recent results, see \cite{BS14, R21, VY20}. In particular, Solomyak~\cite[Theorem 1.3]{Sol21} proved that every non-atomic self-similar measure on $\mathbb R$ has power Fourier decay, provided that the vector of contraction ratios lies outside an exceptional set of Hausdorff dimension zero. We also refer to a recent survey by Sahlsten~\cite{S25}.

In this paper, we prove some results on the distribution of fractional parts of the sequence $(\xi \alpha^{n})$ under various Diophantine conditions on $\xi$ and $\alpha$. We then apply our results to study the Fourier decay of self-similar measures in Section~\ref{sec:app}. In particular, we employ these new results to obtain a generalization of a recent result of Varj\'{u} and Yu~\cite{VY20}; see Theorem~\ref{thm-5} for a precise statement.

Throughout the paper, we will use the capital letter $C(\alpha)$ or $C(\xi,\alpha)$ to denote a positive constant whose exact value may vary at each occurrence. We denote by $\|x\|$ the distance from a real number $x$ to its nearest integer.

\begin{defn}[Base-$b$ Diophantine exponent]\label{defn:base-exponent}
Let \(b\geq 2\) be an integer.  We denote by \(v_b(\xi)\) the supremum of the set of all real numbers \(v\) for which the inequality
\[
        \|b^n\xi\|<(b^n)^{-v}
\]
holds for infinitely many positive integers \(n\).
\end{defn}
The exponent \(v_b(\xi)\) was introduced by Amou and Bugeaud~\cite{AB10} in their study of Diophantine approximation and expansions in integer bases. For every integer $b\geq 2$, we have $v_b(\xi)=0$ for almost all real numbers $\xi$; see \cite[Section~1]{AB10}.

Our first result is the following. 
\begin{thm}\label{thm-1}
Let $\xi\neq 0$ be a real number, and $\alpha=\frac{p}{q}>1$ be a rational number with $\gcd(p,q)=1$ and $q\geq 1$. Assume that $\eta$ is a real number such that $(p-q)|\eta|<1$. Additionally, if $q=1$ (that is, $\alpha=p$ is an integer), suppose that \(v_p(\xi)<\infty\). Then there exists $\delta_1\in(0,\frac{1}{2})$ that depends only on $\alpha$ and $\eta$, and a positive constant $C=C(\alpha,\xi)$,
such that
\begin{equation}\label{eq:thm-1}
\#\left\{1\leq n\leq N: ~\|\xi \alpha^{n}-\eta\|\geq\delta_1\right\}\geq \lfloor C\log N  \rfloor  
\end{equation}
holds for all positive integers $N$.    
\end{thm}

Note that Theorem~\ref{thm-1} is essentially best possible in general. This can be seen from the example $\alpha=2$, $\eta=0$, and $\xi=\sum_{k=1}^{\infty}\,2^{-2^k}$; here the number of exceptional indices has only logarithmic order. On the other hand, when $\alpha$ is an integer, the condition \(v_\alpha(\xi)<\infty\) cannot simply be omitted. Indeed, if $\alpha=2$, $\eta=0$, and $\xi=\sum_{k=1}^{\infty}\,2^{-k!}$, then \(v_2(\xi)=\infty\), and the left-hand side of \eqref{eq:thm-1} is $o(\log N)$ for each $\delta_1 \in (0,\frac{1}{2})$.

We shall generalize Theorem~\ref{thm-1} to the case of algebraic numbers. Throughout the paper, we use the following notations. Let $\alpha>1$ be a real algebraic number and $f(x)=a_dx^{d}+a_{d-1}x^{d-1}+\cdots+a_{1}x+a_{0} \in \Z[x]$ be its minimal polynomial. Denote by $\alpha_{1} (\!=\!\alpha),\alpha_2,\ldots,\alpha_{d}$ the conjugates of $\alpha$. Note that $f(x)$ is irreducible and all conjugates of $\alpha$ are distinct. We also use $L(\alpha)$, $M(\alpha)$, and $h(\alpha)$ to denote the length, Mahler measure, and logarithmic height of $\alpha$, respectively. That is, 
$$ L(\alpha)=\sum_{j=0}^{d}|a_j|,\,\,\,\,  M(\alpha)=|a_d|\prod_{i=1}^{d}\max\{1,|\alpha_i|\},\,\,\, h(\alpha)=\frac{1}{d}\log(M(\alpha)).$$
We also define $\Tilde{L}(\alpha)=|\sum_{j=0}^d a_j|$. Observe that $\Tilde{L}(\alpha)\neq 0$ (otherwise $f(1)=0$, which implies $\alpha=1$).

\begin{defn}[Liouville number]
A real number $\xi$ is said to be Liouville over a number field $\mathbb{K}$, if $\xi\notin \mathbb{K}$ and for any $\tau>0$, there exists $\beta\in \mathbb{K}$ with $h(\beta)>0$, such that
\[|\xi-\beta|<\exp(-\tau h(\beta)).\]
\end{defn}

Our second theorem is the following.

\begin{thm}\label{thm-2}
Let $\alpha>1$ be a real algebraic number. Let $\eta$ be a real number such that $\Tilde{L}(\alpha)|\eta|<1$.  Suppose that $\xi\notin\Q(\alpha)$ is a real number which is not Liouville over $\Q(\alpha)$, then there exists $\delta_2\in(0,\frac{1}{2})$ that depends only on $\alpha$ and $\eta$, and a positive constant $C=C(\alpha,\xi)$, such that
\begin{equation}\label{eq:thm-2}
\#\left\{1\leq n\leq N: ~\|\xi \alpha^{n}-\eta\|\geq\delta_2\right\}\geq \lfloor C\log N  \rfloor    
\end{equation}
holds for all positive integers $N$.
\end{thm}

Note that when $\xi\notin \Q(\alpha)$ is a real algebraic number, the Roth-Leveque theorem (see \cite[Theorem 2.5]{B04}) implies that $\xi$ is not Liouville over $\Q(\alpha)$ and thus Theorem~\ref{thm-2} applies. In particular, Theorem~\ref{thm-2} strengthens a result of Kaneko~\cite[Theorem 2]{K10}, where he proved a lower bound of the same shape for $\#\left\{1\leq n\leq N: ~\{\xi \alpha^{n}\}\geq\delta_2\right\}$ under the extra assumption that $\xi\notin \Q(\alpha)$ is a real algebraic number.

\begin{rem}
Under the extra assumption that $\alpha$ is an algebraic integer and $\xi>0$ is algebraic with
$\xi\notin\Q(\alpha)$, Kaneko \cite[Theorem 1.1]{K10}
showed that there exists $\delta=\delta(\alpha)>0$ such that
\[
\#\left\{1\leq n\leq N:
\{\xi\alpha^n\}\geq\delta\right\}
\geq C(\alpha,\xi)
\frac{(\log N)^{3/2}}{(\log\log N)^{1/2}},
\]
for all sufficiently large $N$, where
$C(\alpha,\xi)>0$ depends on $\alpha$ and the degree 
$[\Q(\alpha,\xi):\Q(\alpha)]$ of the field extension $\Q(\alpha,\xi)/\Q(\alpha)$. One main tool of his proof is a quantitative parametric
subspace theorem by Bugeaud and Evertse \cite{BE08}.
While his result gives a stronger logarithmic bound in
this setting, our proof is different and shorter.
\end{rem}

In the next example, we illustrate the necessity of the assumption $\xi\notin\Q(\alpha)$ in Theorem \ref{thm-2}.

\begin{example}
 Consider the case $\xi=\frac{1}{\sqrt{5}}$, $\eta=0$, and $\alpha=\frac{1+\sqrt{5}}{2}$. Then clearly $\xi\in\Q(\alpha)$. It is well-known that the Fibonacci sequence $(F_n)$ has a closed-form formula
 \[F_n=\frac{1}{\sqrt{5}}\left(\frac{1+\sqrt{5}}{2}\right)^n-\frac{1}{\sqrt{5}}\left(\frac{1-\sqrt{5}}{2}\right)^n.\]
Observe that $|\frac{1-\sqrt{5}}{2}|\in (0,1)$ so that $\|\xi \alpha^{n}\|\rightarrow 0$. Thus, the left-hand side of \eqref{eq:thm-2} remains bounded as $N\to\infty$, so the conclusion of Theorem~\ref{thm-2} fails.
\end{example}
 
Motivated by the above example, we prove the following theorem.
\begin{thm}\label{thm-3}
Let $\alpha>1$ be a real algebraic number with at least one conjugate different from itself lying outside the unit circle,  and $\eta$ be a real number such that $\Tilde{L}(\alpha)|\eta|<1$.  If $\xi\neq0$, then there exists $\delta_3\in(0,\frac{1}{2})$ that depends only on $\alpha$ and $\eta$, and a positive constant $C=C(\alpha,\xi)$, such that
\begin{equation}\label{Calphaxi}
\#\left\{1\leq n\leq N:~\|\xi \alpha^{n}-\eta\|\geq\delta_3\right\}\geq \lfloor C\log N  \rfloor  
\end{equation}
holds for all positive integers $N$. Moreover, if $1\leq \xi \leq \alpha$, then there exists $\delta_4\in(0,\frac{1}{2})$ that depends only on $\alpha$, and a positive constant $C=C(\alpha)$, such that
\begin{equation}\label{Calpha}
\#\left\{1\leq n\leq N:~\|\xi \alpha^{n}\|\geq\delta_4\right\}\geq \lfloor C\log N  \rfloor  
\end{equation}
holds for all positive integers $N$. 
\end{thm}


Since \(\|x-\eta-a\|=\|x-\eta\|\) for every \(a\in\mathbb Z\), the
condition \(\widetilde L(\alpha)|\eta|<1\) in Theorems~\ref{thm-2} and~\ref{thm-3} may be replaced by
\(\widetilde L(\alpha)\|\eta\|<1\). In particular, if
\(\widetilde L(\alpha)=1\), then \(\eta\) is arbitrary. If
\(\widetilde L(\alpha)=2\), then \(\eta\) is arbitrary except when
\(\eta-\frac12\in\mathbb Z\).

\medskip

This paper is organized as follows. In Section~\ref{sec:prelim}, we introduce some preliminaries. Section~\ref{sec:proofs} is devoted to the proofs of Theorems~\ref{thm-1},~\ref{thm-2}, and~\ref{thm-3}. In Section~\ref{sec:app}, we show an application to Fourier analysis of self-similar measures.

\section{Preliminaries}\label{sec:prelim}
Let $\alpha>1$ be a real algebraic number and $f(x)=a_dx^{d}+a_{d-1}x^{d-1}+\cdots+a_{1}x+a_{0}\in \mathbb{Z}[x]$ be its minimal polynomial. Denote the height of $f$ by $\overline{f}=\max_{0\leq i\leq d}|a_{i}|$, and $\alpha=\alpha_{1},\alpha_2,\cdots,\alpha_{d}$ by the conjugates of $\alpha$. Note that $f(x)$ is irreducible and all conjugates of $\alpha$ are distinct.

The following two lemmas can be found in \cite[Chapter VIII, Section 2]{C57}.

\begin{lem}\label{lem-1}
Let \(K\ge 0\). Suppose that a complex sequence \((A_j)\) satisfies
\[
a_dA_{j+d}+a_{d-1}A_{j+d-1}+\cdots+a_0A_j=0
\]
for all \(j=n,n+1,\ldots,n+K-1\). Then there exist complex numbers
\(\theta_1,\ldots,\theta_d\) such that
\[
A_j=\theta_1\alpha_1^j+\cdots+\theta_d\alpha_d^j
\]
for all \(j=n,n+1,\ldots,n+K+d-1\).
\end{lem}

\begin{lem}\label{lem-2}
The system of equations
 \begin{equation*}
X_1\alpha_1^{n}+X_2\alpha_2^{n}+\cdots+X_d\alpha_d^{n}=Y_{n},\quad \mbox{for all}\;n=0,1,\ldots,d-1,
 \end{equation*}
has a unique solution
  \begin{equation*}
X_{n}=\frac{1}{G_0(\alpha_n)}\sum_{k=0}^{d-1}\beta_{n,k}Y_k,\quad \mbox{for all}\;n=1,\ldots,d,
 \end{equation*}
where $G_0(x)=\sum_{m=1}^{d}ma_m x^{m-1}$ and $\beta_{n,k}=\sum_{m=k+1}^{d}a_m\alpha_n^{m-k-1}$. 
\end{lem}

The following lemma lists some basic properties of the height of algebraic numbers; see \cite[Proposition 1.2]{Za18}.
\begin{lem}\label{lem:height}
If $\beta_1,\ldots,\beta_n$ are algebraic numbers, then 
\begin{enumerate}[(i)]
\item  $h(\beta_1+\cdots+\beta_n)\leq h(\beta_1)+\cdots+h(\beta_n)+\log n$.
\item $h(\beta_1\cdots\beta_n)\leq h(\beta_1)+\cdots+h(\beta_n)$.
\item $h(\beta^n)=|n|h(\beta)$ for any algebraic number $\beta$ and $n\in\mathbb{Z}$.
\end{enumerate}
\end{lem}

The following elementary separation estimate is a standard variant of
Garsia's separation lemma \cite[Lemma 1.51]{G62}. We include a short proof for the sake of completeness.

\begin{lem}\label{lem:Garsia}
Let \(\lambda\) be an algebraic number with \(|\lambda|>1\), and let
\(\lambda=\lambda_1,\lambda_2,\ldots,\lambda_s\) be its conjugates. Let
$p(x)=c_mx^{m}+c_{m-1}x^{m-1}+\cdots+c_{1}x+c_{0} \in \mathbb{Z}[x]$ be a polynomial. If $p(\lambda)\neq 0$, then 
$$
|p(\lambda)|\geq \overline{p}^{\,1-s}(m+1)^{1-s}M(\lambda)^{-m}.
$$
\end{lem}
\begin{proof}
Let \(F(x)=a_s\prod_{i=1}^s(x-\lambda_i)\in\Z[x]\) be the minimal
polynomial of \(\lambda\). Since \(p(\lambda)\neq0\), the resultant
\(\operatorname{Res}(F,p)\) is a non-zero rational integer. Hence
\[
        1\le |\operatorname{Res}(F,p)|
        = |a_s|^m\prod_{i=1}^s |p(\lambda_i)|.
\]
For \(2\le i\le s\), we have
\[
        |p(\lambda_i)|
        \le \overline p\sum_{j=0}^m|\lambda_i|^j
        \le \overline p(m+1)\max\{1,|\lambda_i|\}^m.
\]
Therefore,
\[
        |p(\lambda)|
        \ge
        |a_s|^{-m}\overline p^{\,1-s}(m+1)^{1-s}
        \prod_{i=2}^s\max\{1,|\lambda_i|\}^{-m}.
\]
Since \(|\lambda|>1\), we have
\[
        M(\lambda)
        =
        |a_s||\lambda|\prod_{i=2}^s\max\{1,|\lambda_i|\}
        \ge
        |a_s|\prod_{i=2}^s\max\{1,|\lambda_i|\}.
\]
Thus
\[
        |p(\lambda)|
        \ge
        \overline p^{\,1-s}(m+1)^{1-s}M(\lambda)^{-m},
\]
as desired.
\end{proof}

\section{Proofs of Theorems}\label{sec:proofs}
In this section, we shall provide proofs of our theorems. 

One ingredient of our proofs is the following combinatorial lemma. A version of the lemma has appeared in \cite[Proposition 2.1]{GM17}. Here we present a simple proof. 
\begin{lem}\label{lem-3}
Let $\xi, \alpha, \eta$ be fixed real numbers. Let $0<\delta<1$ be a constant. Suppose that there exist constants $\gamma,\gamma_{0}>0$ such that
 \begin{equation*}
\sup\{k: \|\xi\alpha^n-\eta\|<\delta,\|\xi\alpha^{n+1}-\eta\|<\delta,\ldots,\|\xi\alpha^{n+k}-\eta\|<\delta\}<\gamma n+\gamma_{0}
\end{equation*}
holds for all $n \in \N$. If $\gamma_0/\gamma \leq S$, then there is a positive constant $C$ depending on $\gamma$ and $S$, such that
\[\#\left\{1\leq n\leq N:~ \|\xi \alpha^{n}-\eta\|\geq\delta\right\}\geq \lfloor C\log N  \rfloor  \]
holds for all $N \in \N$.
\end{lem}

\begin{proof}
Let $X=\{n \in \mathbb{N}: \|\xi \alpha^{n}-\eta\|\geq\delta\}$ and label the elements in $X$ in increasing order by $x_1, x_2, \cdots$. Let $x_0=0$. By the assumption, for each $m \geq 0$, we have $$x_{m+1}\leq x_m+\gamma(x_m+1)+\gamma_0+2=x_m(1+\gamma)+(\gamma+\gamma_0+2).$$
It follows that $x_{m+1}+B\leq (1+\gamma)(x_m+B)$, where $B=(\gamma+\gamma_0+2)/\gamma$. Thus, $x_m+B\leq B(1+\gamma)^m $ for all $m \in \mathbb{N}$. Since $B \leq S+1+2/\gamma$,  the lemma follows readily.
\end{proof}

By Lemma~\ref{lem-3}, in order to prove our main results, it suffices to show that the number of maximum consecutive blocks of $(\xi\alpha^{n})$ falling into a small region cannot be too large. This will be achieved by various tools from Diophantine approximation. 

We first present the proof of Theorem \ref{thm-1} for the special case $\alpha=\frac{p}{q}$ is rational, before moving to the more technical proofs for general algebraic numbers.

\begin{proof}[Proof of Theorem \ref{thm-1}]
Define $A_{n}=A_n\left(\xi,\alpha,\eta \right)$ and $\varepsilon_{n}=\varepsilon_n\left(\xi,\alpha,\eta \right)$ as follows:
\[\xi\alpha^n=A_{n}+\eta+\varepsilon_{n}, \qquad A_{n}\in\mathbb{Z},\quad-\frac{1}{2}<\varepsilon_{n}\leq\frac{1}{2}.\]
Observe that $|\varepsilon_{n}|=\|\xi\alpha^n-\eta\|$.
Consider
\begin{align*}
    \xi\alpha^{n}=A_{n}+\eta+\varepsilon_{n},\qquad& \xi\alpha^{n+1}=A_{n+1}+\eta+\varepsilon_{n+1}.
\end{align*}
Choose $\delta_{1}:=\frac{1-(p-q)|\eta|}{p+q}$. By the assumption $(p-q)|\eta|<1$, we have $0<\delta_{1}\leq \frac{1}{p+q}$. 

If $\max\left\{|\varepsilon_{n}|,|\varepsilon_{n+1}|\right\}<\delta_{1}$, we have
\begin{equation}\label{(2.4)}
|pA_{n}-qA_{n+1}|\leq (p-q)|\eta|+|p\varepsilon_{n}-q\varepsilon_{n+1}|<(p-q)|\eta|+(p+q)\delta_{1}= 1.    
\end{equation}
Since the left-hand side of inequality~\eqref{(2.4)} is a rational integer, it follows that
\[qA_{n+1}=pA_{n}.\]
\par Similarly, if there are $k$ consecutive indices followed by the index $n$ such that
$$\max\left\{|\varepsilon_{n}|,|\varepsilon_{n+1}|,\ldots,|\varepsilon_{n+k}|\right\}<\delta_{1},$$
then
\begin{equation}\label{(2.5)}
q^kA_{n+k}=p^{k} A_{n}.
\end{equation}
Next, we consider two cases according to whether $q=1$. In both cases, we show that $k$ is at most linear in $n$, and thus the statement of the theorem follows from Lemma \ref{lem-3}.

\textbf{Case 1: } $q=1$, that is, $\alpha=p$. Equation~\eqref{(2.5)} implies that $A_{n+k}=p^k A_n$. Hence
\[
        |p^n\xi-A_n|
        =
        p^{-k}|p^{n+k}\xi-A_{n+k}|
        =
        p^{-k}|\eta+\varepsilon_{n+k}|
        \leq 2p^{-k}.
\]
In particular,
\begin{equation}\label{eq:vp-upper}
        \|p^n\xi\|\leq 2p^{-k}.
\end{equation}
Let \(v=v_p(\xi)<\infty\). By the definition of \(v_p(\xi)\), for every \(\epsilon>0\), and after decreasing the constant to deal with finitely many exceptional values of \(n\), there exists \(C(\xi,\epsilon)>0\) such that
\begin{equation}\label{eq:vp-lower}
        \|p^n\xi\|\geq C(\xi,\epsilon)p^{-n(v+\epsilon)}
\end{equation}
for all positive integers \(n\). Here \(v_p(\xi)<\infty\) also excludes the possibility that
\(\|p^n\xi\|=0\) for some \(n\), since then \(\|p^m\xi\|=0\) for all
\(m\ge n\), and hence \(v_p(\xi)=\infty\). Combining inequalities~\eqref{eq:vp-upper} and \eqref{eq:vp-lower}, we obtain
\[
        C(\xi,\epsilon)p^{-n(v+\epsilon)}\leq 2p^{-k}.
\]
It follows that
\[
        k\leq (v_p(\xi)+\epsilon)n+\frac{\log 2-\log C(\xi,\epsilon)}{\log p}.
\]

\textbf{Case 2: }$q\geq 2$. Since $\gcd(p,q)=1$, equation~\eqref{(2.5)} implies that $q^k$ divides $A_n$ and thus $|A_n|\geq q^k$ (assuming $n$ is sufficiently large). It follows that $q^k \leq |A_n|\leq |\xi|\alpha^n +2$ and thus $$k \leq \frac{n\log \alpha}{\log q}+C(\xi, \alpha),$$
as required.
\end{proof}

Next, we prove Theorem \ref{thm-2}.

\begin{proof}[Proof of Theorem \ref{thm-2}]
Define $A_{n}=A_n\left(\xi,\alpha,\eta \right)$ and $\varepsilon_{n}'=\varepsilon_n'\left(\xi,\alpha,\eta \right)$ as follows:
\[\xi\alpha^n=A_{n}+\eta+\varepsilon_{n}', \qquad A_{n}\in\mathbb{Z},\quad-\frac{1}{2}<\varepsilon_{n}'\leq\frac{1}{2}.\]
Set $\varepsilon_{n}=\eta+\varepsilon_{n}'$ so that $A_{n}=\xi\alpha^n-\varepsilon_{n}$ and $|\varepsilon_{n}|< 2$. 

Since $\alpha$ is an algebraic number with minimal polynomial $f(x)$, it follows that
\[\xi\alpha^{n}(a_{d}\alpha^{d}+a_{d-1}\alpha^{d-1}+\cdots+a_{1}\alpha+a_{0})=0.\]
Thus
\begin{equation}\label{2.8}
|a_{d}A_{n+d}+a_{d-1}A_{n+d-1}+\cdots+a_{0}A_{n}|=|a_{d}\varepsilon_{n+d}+a_{{d-1}}\varepsilon_{n+d-1}+\cdots+a_{0}\varepsilon_{n}|.    
\end{equation}
Take $\delta_{2}:=\frac{1-|\eta|\widetilde{L}(\alpha)}{L(\alpha)}$. 
By the assumption $|\eta|\widetilde{L}(\alpha)<1$, we have $0<\delta_2\leq \frac{1}{L(\alpha)}$. Assume that
\begin{equation}\label{2.9}
\max\left\{|\varepsilon_{n}'|,|\varepsilon_{n+1}'|,\ldots,|\varepsilon_{n+d}'|\right\}<\delta_{2},
\end{equation}
then equation~\eqref{2.8} and inequality~\eqref{2.9} imply that
$$|a_{d}A_{n+d}+a_{d-1}A_{n+d-1}+\cdots+a_{0}A_{n}|< |\eta|\widetilde{L}(\alpha)+\delta_{2}(|a_{0}|+|a_{1}|+\cdots+|a_{d}|)=|\eta|\widetilde{L}(\alpha)+\delta_2 L(\alpha)=1.$$
As before, the left-hand side is a rational integer, thus it is equal to zero. That is,
\begin{equation}\label{2.10}
a_{d}A_{n+d}+a_{d-1}A_{n+d-1}+\cdots+a_{0}A_{n}=0.
\end{equation}

In the following discussion, assume that there are \(k+2\) consecutive indices starting from the index
\(n\) such that
\[
\max \left\{
|\varepsilon'_{n}|, |\varepsilon'_{n+1}|,\ldots,|\varepsilon'_{n+k+1}|
\right\}<\delta_2.
\]
If \(k<d-2\), then there is nothing to prove. Hence we may assume that
\(k\ge d-2\). Then equation~\eqref{2.10} holds for
\[
        j=n,n+1,\ldots,n+k+1-d.
\]
Applying Lemma~\ref{lem-1} with \(K=k+2-d\), we obtain
complex numbers \(\theta_1,\theta_2,\ldots,\theta_d\) such that
\[
A_{n+i}=\theta_1\alpha_1^{n+i}+\theta_2\alpha_2^{n+i}
+\cdots+\theta_d\alpha_d^{n+i},
\quad 0\le i\le k+1.
\]
Applying Lemma \ref{lem-2} with
$X_1=\theta_1\alpha^{n},\ldots,X_d=\theta_d\alpha_d^{n}$, and $Y_i=A_{n+i},$ we have
\begin{equation}\label{2.11}
\theta_1\alpha^{n}=\frac{G_1(\alpha)}{G_0(\alpha)}.
\end{equation}
where \[G_0(\alpha)=\sum_{m=1}^{d}ma_m \alpha^{m-1}, \quad G_1(\alpha)=\sum_{i=0}^{d-1}\sum_{m=i+1}^{d}a_m\alpha^{m-i-1}A_{n+i}.\]
Since $a_m$ is a rational integer for each $0 \leq m \leq d$ and $A_{n+i}$ is a rational integer for each $0 \leq i \leq k$, it follows that $G_0(\alpha)$ and $G_1(\alpha)$ are both in $\Q(\alpha)$. Thus, equation~\eqref{2.11} implies that $\theta_1\in\Q(\alpha)$. 

We claim that
\begin{equation}\label{2.12}
h(\theta_1)\leq \bigg(\frac{1}{2}(d^2+d)\log\alpha+h(\alpha)\bigg )n+C(\xi,\alpha).
\end{equation}
\medskip
Indeed, using Lemma~\ref{lem:height},
\begin{equation}\label{2.13}
h(\theta_1)=h\bigg(\frac{G_1(\alpha)}{G_0(\alpha)\alpha^{n}}\bigg)\leq h\bigg(\frac{G_1(\alpha)}{G_0(\alpha)}\bigg)+h(\alpha^{-n})\leq nh(\alpha)+h({G_1(\alpha)})+h({G_0(\alpha)}).
\end{equation}
By Lemma \ref{lem-2} and Lemma~\ref{lem:height},
\begin{equation}\label{2.14}
h({G_0(\alpha)})=h\bigg(\sum_{m=1}^{d}ma_m \alpha^{m-1}\bigg)\leq\frac{d(d-1)}{2}h(\alpha)+\sum_{m=1}^{d}\log (m\max\{1,|a_m|\})+\log d.
\end{equation}
Observe that for \(0\le i\le d-1\), we have
\[
        |A_{n+i}|\le |\xi|\alpha^{n+i}+2
        \le |\xi|\alpha^{n+d-1}+2.
\]
Thus
\begin{align}
h(G_1(\alpha))
&=h\bigg(\sum_{i=0}^{d-1}\sum_{m=i+1}^{d}a_m\alpha^{m-i-1}A_{n+i}\bigg) \notag\\
&\le
\sum_{i=0}^{d-1}\sum_{m=i+1}^d
\bigl(\log\max\{1,|a_m|\}+(m-i-1)h(\alpha)\bigr)  \notag\\
&\quad+
\frac{d(d+1)}2\log\bigl(|\xi|\alpha^{n+d-1}+2\bigr)
+\log\frac{d(d+1)}2 \notag\\
&\le \frac{(d^2+d)\log\alpha}{2}n+C(\alpha,\xi). \label{2.15}
\end{align}
From inequalities~\eqref{2.13},~\eqref{2.14}, and~\eqref{2.15}, we obtain the claim \eqref{2.12}.

On the other hand, for each $0 \leq i \leq k+1$, from the definition of $A_{n+i}$, we have
$$
A_{n+i}=\xi\alpha^{n+i}-\varepsilon_{n+i}=\xi\alpha_1^{n+i}-\varepsilon_{n+i}=\theta_1\alpha_1^{n+i}+\theta_2\alpha_2^{n+i}+\cdots+\theta_d\alpha_d^{n+i},$$
and thus
$$\varepsilon_{n+i}=(\xi-\theta_1)\alpha_1^{n+i}-\theta_2\alpha_2^{n+i}-\cdots-\theta_d\alpha_d^{n+i}.$$
We consider the last $d$ equations
\begin{equation*}
\varepsilon_{n+l}=(\xi-\theta_1)\alpha^{n+l}-\theta_2\alpha_2^{n+l}-\cdots-\theta_d\alpha_d^{n+l}, \quad l=k-d+2,\ldots,k+1.
\end{equation*}
equivalently,
\begin{equation*}
\varepsilon_{n\!+k-\!d+\!2+s}=(\xi-\theta_1)\alpha^{n+k-d+2}\alpha^s-\theta_2\alpha_2^{n+k-d+2}\alpha_2^s-\cdots-\theta_d\alpha_d^{n+k-d+2}\alpha_d^s, \quad s=0,\ldots,d-1.
\end{equation*}
Now apply Lemma \ref{lem-2} with
$$Y_s=\varepsilon_{n+k-d+2+s},\,\,\, X_1=(\xi-\theta_1)\alpha^{n+k-d+2},\,\,\, X_2=-\theta_2\alpha_2^{n+k-d+2},\ldots,X_d=-\theta_d\alpha_d^{n+k-d+2},$$
we get
\begin{equation}\label{2.17}
X_1=(\xi-\theta_1)\alpha^{n+k-d+2}=\frac{G_2(\alpha)}{G_0(\alpha)}.
\end{equation}
where
\begin{equation}\label{G2}
G_2(\alpha)=\sum_{s=0}^{d-1}\sum_{m=s+1}^{d}a_m\alpha^{m-s-1}\varepsilon_{n+k-d+2+s}.
\end{equation}
Applying the trivial bound $|\varepsilon_{n+k-d+2+s}|\leq2$ in equation~\eqref{G2}, inequality~\eqref{2.17} then implies that
\begin{equation}\label{2.18}
|\xi-\theta_1|\leq C(\alpha)\frac{1}{\alpha^{n+k}}\leq C(\alpha)e^{-(n+k)\log \alpha}.
\end{equation}
For sufficiently large $n$, inequality~\eqref{2.18} implies that $h(\theta_1)>0$; indeed, if $h(\theta_1)=0$, then $\theta_1$ belongs to a fixed finite set of real algebraic numbers in $\Q(\alpha)$ of height zero, whereas the right-hand side of~\eqref{2.18} tends to zero and $\xi\notin\Q(\alpha)$. By the assumption that $\xi\notin\Q(\alpha)$ and $\xi$ is not Liouville over $\Q(\alpha)$, it follows that there exists $\tau=\tau(\xi)<\infty$ such that
\[|\xi-\theta_1|>\exp(-\tau h(\theta_1)).\]

By inequality~\eqref{2.12}, it follows that
\begin{equation}\label{2.19}
|\xi-\theta_1|>C'(\xi,\alpha)\exp\bigg(-\tau\bigg(\frac{1}{2}(d^2+d)\log\alpha+h(\alpha)\bigg ) n\bigg).
\end{equation}
Combining inequality~\eqref{2.18} with inequality~\eqref{2.19} yields
\[k<\bigg(\tau\bigg(\frac{d^2+d}{2}+\frac{h(\alpha)}{\log\alpha}\bigg)-1\bigg)n+C''(\xi,\alpha).\]
Enlarging \(\tau\) if necessary, we may assume that the coefficient of \(n\) is positive.
Now the theorem follows from Lemma \ref{lem-3}.
\end{proof}

Finally, we use a similar strategy to prove Theorem \ref{thm-3}.

\begin{proof}[Proof of Theorem \ref{thm-3}]
The proof is similar to that of Theorem~\ref{thm-2}, and we shall follow the same notations. By the assumption, there exists a conjugate of $\alpha$ lying outside the unit circle. Without loss of generality, say $|\alpha_2|>1$.\par
Similar to the proof of Theorem~\ref{thm-2}, Lemma \ref{lem-2} implies that
\begin{equation}\label{2.20}
\theta_2\alpha_2^{n}=\frac{G_1(\alpha_2)}{G_0(\alpha_2)},
\end{equation}
\begin{equation}\label{2.21}
\theta_2\alpha_2^{n+k-d+2}=-\frac{G_2(\alpha_2)}{G_0(\alpha_2)}.
\end{equation}
Recall that $G_0(\alpha)=\sum_{m=1}^{d}ma_m \alpha^{m-1}$ is a non-zero constant only depending on $\alpha$. Recall also that $G_1(x)=\sum_{i=0}^{d-1}\sum_{m=i+1}^{d}a_mx^{m-i-1}A_{n+i}$ is a polynomial with integral coefficients of degree at most $d-1$. For all sufficiently large $n$, the integers $A_n,A_{n+1},\ldots,A_{n+d-1}$ are not all zero. Hence $G_1(x)$ is not the zero polynomial, and since the minimal polynomial of $\alpha_2$ has degree $d$, we have $G_1(\alpha_2)\neq 0$. The finitely many smaller values of $n$ can be absorbed into the final constant.

As before, $|A_{n+i}|\leq |\xi\alpha^{n+i}|+2\leq |\xi|\alpha^{n+d-1}+2$ for $0\leq i\leq d-1$. It follows that 
\begin{equation}\label{G1}
\overline{G_1}\leq d\,\overline{f}\bigl(|\xi|\alpha^{n+d-1}+2\bigr).    
\end{equation}
Applying Lemma~\ref{lem:Garsia} with \(p(x)=G_1(x)\) and
\(\lambda=\alpha_2\), and using \(\deg G_1\le d-1\), we obtain
\begin{equation}\label{2.22}
        |G_1(\alpha_2)|
        \ge
        \overline{G_1}^{\,1-d}d^{1-d}M(\alpha_2)^{-(d-1)}
        \ge C_1(\xi,\alpha_2)\alpha^{-n(d-1)}.
\end{equation}

From equation~\eqref{2.20} and inequality~\eqref{2.22}, we have the lower bound
\begin{equation}\label{2.23}
|\theta_2| \geq C_2(\alpha_2)|\alpha_2|^{-n}|G_1(\alpha_2)|\geq C_3(\xi,\alpha_2)|\alpha_2|^{-n}\alpha^{-n(d-1)}.
\end{equation}
On the other hand, from equation~\eqref{2.21}, the observation that $|\varepsilon_{n+k-d+2+s}|\leq 2$, and the definition of $G_2(\alpha_2)$ in equation~\eqref{G2}, we have the upper bound
\begin{equation}\label{2.24}
|\theta_2|\leq C_4(\alpha_2)|\alpha_2|^{-n-k+d-2}.
\end{equation}
Combining inequalities~\eqref{2.23} with \eqref{2.24}, we obtain that
\begin{equation}\label{kbound}
k<\frac{(d-1)\log \alpha}{\log |\alpha_2|}n+\frac{\log C_5(\xi,\alpha_2)}{\log |\alpha_2|}.    
\end{equation}
Since $d>1$, inequality~\eqref{Calphaxi} follows from Lemma~\ref{lem-3}. 

Next, assume additionally that $1 \leq \xi\leq \alpha$. We aim to show that the constant $C$ on the right-hand side of inequality~\eqref{Calphaxi} does not depend on $\xi$. Given Lemma~\ref{lem-3} and inequality~\eqref{kbound}, it suffices to show that the constant $C_5$ does not depend on $\xi$. To achieve that, we perform the same analysis more carefully. Since $1\leq \xi \leq \alpha$, inequality~\eqref{G1} implies that the constant $C_1$ in inequality~\eqref{2.22} does not depend on $\xi$, and thus the constant $C_3$ in inequality~\eqref{2.23} does not depend on $\xi$. Now it is clear that the constant $C_5$ in inequality~\eqref{kbound} does not depend on $\xi$.
\end{proof}

\section{Applications}\label{sec:app}
In this section, we present an application of our results to Fourier transforms of some self-similar measures. In order to state the results, we first introduce some notations.

Recall that an iterated function system (IFS) \(\!\mathscr{I}=\!\left\{f_1,f_2,\ldots,f_m\right\},m\geq2\) is a finite family of strict contraction mappings on $\mathbb{R}$. Let $P=(p_1, p_2,\ldots, p_m)$ be a non-degenerate probability vector, that is, $\sum_{i=1}^{m}p_i=1$ and $0< p_i<1\mbox{ for all } 1\leq i\leq m$. It is well-known that there is a unique nonempty compact set $K\subset \mathbb{R} $ such that
 \[K=\bigcup_{i=1}^{m}f_{i}(K)\]
and a unique Borel probability measure $\mu$ on $K$ satisfying
$$\mu=\sum_{i=1}^{m} p_{i}f_{i}\mu,$$
where $f\mu:=\mu\circ f^{-1}$ is the push-forward of $\mu$ by the transformation $f$; see for example \cite{H81}. We say that $K$ is a self-similar set and $\mu$ is the self-similar measure on $K$ associated with the probability vector $P$.\par
Suppose that we have the IFS which has the following form
\[f_i(x)=r^{l_i}x+a_i, 1\leq i \leq m,\]
where $0<r<1$, each $l_i$ is a positive integer with $\gcd(l_1,l_2\ldots,l_m)=1$, and $a_i\in\mathbb{R}$. By \cite[Lemma 4.1]{VY20}, without loss of generality, we assume that $l_1=l_2$ in the following discussion. We can also assume that $a_1>a_2$.

Recall that the Fourier transform of a measure $\mu$ is
\[\widehat{\mu}(u):=\int_{\R} \exp(2\pi i u t)\,d\mu(t), \quad\forall u \in \R.\] \\
The following theorem is due to Varj\'{u}-Yu \cite[Theorem 1.5]{VY20}:
\begin{thm}[Varj\'{u}-Yu]\label{thm-4}
Let $\mu$ be the non-atomic self-similar measure corresponding to the IFS as above.  Then
$$|\widehat{\mu}(u)|\leq\exp\Big(-C^{-1}\sum_{n>C}\|(a_1-a_2)u r^{n}\|^2\Big)$$ where $C>0$ is a constant depending only on $l_1, l_2,\ldots, l_m$ and $p_1, p_2,\ldots, p_m$.
\end{thm}
As a corollary, they deduced that if $r^{-1}$ is an algebraic integer that is not a Pisot or Salem number, then $|\widehat{\mu}(u)|=O((\log |u|)^{-\gamma})$ as $|u|\to\infty$ for some constant $\gamma>0$ \cite[Corollary 1.6]{VY20}. They showed that this corollary follows from Theorem~\ref{thm-4} and \cite[Proposition 5.5]{BS14} by Bufetov and Solomyak.

We shall extend their result to the case that $r^{-1}$ is a general algebraic number in the following theorem, provided that $r^{-1}$ has at least one conjugate different from itself outside the unit circle. Note that when $r^{-1}$ is an algebraic integer, by definition, such a condition coincides with the assumption that $r^{-1}$ is not a Pisot or Salem number in their corollary. Thus, our theorem extends their result.

\begin{thm}\label{thm-5}
Let $\mu$ be the non-atomic self-similar measure corresponding to the IFS as above. Suppose $r^{-1}$ is an algebraic number which has at least one conjugate different from itself outside the unit circle. Then
$$|\widehat{\mu}(u)|=O((\log |u|)^{-\gamma})\quad (|u|\to\infty),$$ where $\gamma>0$ is a constant.
Furthermore, $\mu$-almost all $x$ is normal to any base $b\geq 2,b\in\mathbb{Z}$.
\end{thm}

\begin{proof}
Let $C$ be the constant from Theorem~\ref{thm-4} and set $C_0=\lfloor C \rfloor$. Let $\alpha=r^{-1}$ so that $\alpha>1$. Since $|\widehat\mu(-u)|=|\widehat\mu(u)|$, it suffices to consider positive $u$ large enough. Let $k$ be the unique integer such that 
$$\alpha^{k}\leq (a_1-a_2)u r^{C_0} <\alpha^{k+1};$$
then we can write
$$(a_1-a_2)u r^{C_0}=\xi \alpha^{k},$$ where $1\leq \xi <\alpha$. By Theorem \ref{thm-4}, 
$$|\widehat{\mu}(u)|\leq\exp\Big(-C^{-1}\sum_{n=C_0+1}^{\infty}\|\xi \alpha^{k-n+C_0}\|^2\Big)\leq\exp\Big(-C^{-1}\sum_{j=1}^{k-1}\|\xi \alpha^{j}\|^2\Big).$$
Note that $k \asymp \log u$, by the second part of Theorem \ref{thm-3}, 
$$
\sum_{j=1}^{k-1}\|\xi \alpha^{j}\|^2 \geq C_1(\alpha) \log (k-1) \cdot \delta(\alpha)^2  \geq C_2(\alpha)\log \log u.
$$
Therefore, we have
$$
|\widehat{\mu}(u)|\leq (\log u)^{-C^{-1}C_2(\alpha)},
$$
proving the first statement with $\gamma=C^{-1}C_2(\alpha)>0$.

The second statement is a consequence of the first statement and the celebrated Davenport-Erd\H{o}s-Leveque theorem \cite{DEL}. More precisely, the first statement implies that $$\sum_{n=2}^{\infty} \frac{|\widehat{\mu}(n)|}{n \log n} \ll \sum_{n=2}^{\infty} \frac{1}{n (\log n)^{1+\gamma}}<\infty,$$ thus the second statement follows from Lyons' theorem \cite[Theorem 4]{L86}. 
\end{proof}

\begin{rem}
In the second part of Theorem~\ref{thm-3}, one may
expect to improve the lower bound from logarithmic
to a power of $N$. Such an improvement to $N^\theta$
would strengthen Theorem~\ref{thm-5} to
$|\widehat\mu(u)|\ll\exp(-c(\log|u|)^\theta)$,
giving polynomial Fourier decay when $\theta=1$.
However, obtaining such an improvement appears
difficult without additional assumptions; see
Baker and Banaji~\cite{BB25} for related results
in a different setting. One may also consider
higher-dimensional extensions to self-affine
measures on $\R^d$; see~\cite{LS20,R21,BKS24}.

Note that the Davenport--Erd\H{o}s--LeVeque
criterion~\cite{DEL} can be used to extend the
normality conclusion of Theorem~\ref{thm-5}
from integer bases $b\geq2$ to any fixed real
$\beta>1$: the sequence $(x\beta^n)_{n\geq0}$
is uniformly distributed modulo one for
$\mu$-almost every $x$; see also~\cite{TZ25}.
For related quantitative discrepancy results,
see~\cite{PVZZ22}.
\end{rem}

\section*{Acknowledgments}
The first author thanks Kaneko Hajime, Malabika Pramanik, and Peter Varj\'{u} for helpful comments and suggestions. The second author thanks Greg Knapp
for helpful discussions. The authors thank the anonymous referees for their careful reading and helpful suggestions.  The first author is funded by a scholarship from the China Scholarship Council, Natural Science Foundation of Hubei Province No.2022CFB093, and NSFC grant No.12426661, 
and also enjoyed the hospitality of the University of British Columbia during which this work was carried out.

\end{document}